\documentclass{amsart}
\usepackage{amsmath}
\usepackage{amssymb}
\usepackage{amsthm}
\usepackage{bbm}
\usepackage{cite}
\usepackage{leftidx}
\usepackage{pgf}

\usepackage{tikz-cd}
\usepackage{verbatim}
\usepackage{xypic}

\usepackage{hyperref}

\title{A spectrum-level Hodge filtration on topological Hochschild homology}
\author{Saul Glasman}

\newcommand{\angs}[1]{\langle #1 \rangle}
\newcommand{\Alg}{\mathbf{Alg}}

\newcommand{\bb}{\mathbb}

\newcommand{\colim}[1]{\underset{#1}{\text{colim }}}

\newcommand{\D}{\Delta}

\newcommand{\Fib}{\text{Fib}}
\newcommand{\Fun}{\text{Fun}}

\newcommand{\ho}{\text{ho}}
\newcommand{\Hom}{\text{Hom}}
\newcommand{\id}{\text{id}}
\newcommand{\im}{\text{im }}
\newcommand{\inc}{\subseteq}
\newcommand{\inj}{\hookrightarrow}

\newcommand{\iy}{\infty}

\newcommand{\llim}[1]{\underset{#1}{\text{lim }}}
\newcommand{\mb}{\mathbf}
\newcommand{\mc}{\mathcal}
\newcommand{\mf}{\mathfrak}

\newcommand{\Nat}{\text{Nat}}
\newcommand{\ol}{\overline}

\newcommand{\op}{\text{op}}
\newcommand{\os}{\overset}
\newcommand{\ot}{\otimes}
\newcommand{\si}{\sigma}
\newcommand{\Si}{\Sigma}
\newcommand{\sSet}{\text{sSet}}
\newcommand{\Set}{\text{Set}}
\newcommand{\Sp}{\textbf{Sp}}
\newcommand{\td}{\widetilde}
\newcommand{\toe}{\overset{\sim} \to}
\newcommand{\twa}{\widetilde{\mathcal{O}}}
\newcommand{\Top}{\textbf{Top}}
\newcommand{\ul}{\underline}

\newcommand{\X}{\times}

\newcommand{\Comm}{\text{Comm}}

\newcommand{\CAlg}{\textbf{CAlg}}

\newcommand{\CRing}{\textbf{CRing}}
\newcommand{\F}{\mc{F}}

\newcommand{\Mod}{\mb{Mod}}

\newcommand{\Ran}{\text{Ran}}

\newcommand{\THH}{\text{THH}}

\theoremstyle{definition}

\newtheorem{cor}[subsection]{Corollary}
\newtheorem{dfn}[subsection]{Definition}
\newtheorem{exa}[subsection]{Example}

\newtheorem{lem}[subsection]{Lemma}
\newtheorem{prop}[subsection]{Proposition}
\newtheorem{rem}[subsection]{Remark}

\newtheorem{thm}[subsection]{Theorem}

\begin{document}
\begin{abstract}
We define a functorial spectrum-level filtration on the topological Hochschild homology of any commutative ring spectrum $R$, and more generally the factorization homology $R \ot X$ for any space $X$, echoing algebraic constructions of Loday and Pirashvili. We investigate the properties of this filtration and show that it breaks $\THH$ up into common eigenspectra of the Adams operations. 
\end{abstract}
\maketitle

\section{Introduction}
In this paper, we seek to begin the development of a theory of functorial, spectrum-level Hodge filtrations on topological Hochschild homology ($\THH$) of commutative ring spectra and its cousins, such as $TR$ and $TC$.

What is a Hodge filtration? The original Hodge filtration is a filtration on the de Rham complex of a commutative ring $R$

\[\Omega^*(R) \to \cdots \to (\Omega^*(R))^{\leq n} \to \cdots \to (\Omega^*(R))^{\leq 0} = R\]
where the complex $(\Omega^*(R))^{\leq n}$ is the ``stupid truncation" of $\Omega^*(R)$, which is not a homotopy invariant of $\Omega^*(R)$ but depends on $R$ itself: it has the same terms as $\Omega^*(R)$ in degrees at most $n$, and is zero in degrees greater than $n$. This filtration gives rise to a spectral sequence known as the \emph{Hodge to de Rham spectral sequence}. In a celebrated sequence of papers beginning with \cite{Del71}, Deligne proved that for a $\bb{C}$-algebra $R$, the spectral sequence degenerates at $E_2$, giving rise to a \emph{mixed Hodge structure} on the de Rham cohomology of $R$.

In the special case where $R$ is regular, the Hochschild-Kostant-Rosenberg theorem \cite{HKR} gives an isomorphism of graded abelian groups
\[HH_*(R) \cong \Omega^*(R)\]
where $HH_*$ is Hochschild homology. In \cite{Lod89}, Loday interpreted the Hodge filtration as the $\gamma$-filtration associated to a $\lambda$-ring structure that exists on the Hochschild homology of an arbitrary commutative ring $R$, thus generalizing the Hodge filtration in an algebraic direction. Rationally, this filtration is canonically split, and it coincides with the decomposition by eigenspaces of the Adams operations that exist on any rational $\lambda$-ring.

Later, Pirashvili \cite{Pir00}, working rationally, used functor homology to give a generalization of the rational Hodge decomposition to what he called ``higher-order Hochschild homology". In the present work, we take Pirashvili's approach and run with it in a homotopy-theoretic direction, using a homotopy coend formula for topological Hochschild homology and its ``higher-order" variants to give spectrum-level filtrations with properties analogous to those of Loday and Pirashvili's original examples.

We work with quasicategories throughout, and we will liberally reference Lurie's blockbuster volumes \cite{HTT} and \cite{HA}. The structure of the paper is as follows. In Section \ref{ends}, we collect some useful facts about ends and coends in the $\iy$-categorical context. In Section \ref{tensor}, we discuss tensor products of commutative ring spectra with spaces. In Section \ref{Hodge}, we define our central object of study: for each simplicial set $X$, we define a Hodge-like filtration on the functor

\[\CRing \to \CRing : R \mapsto R \ot X\]

where $\CRing$ is the category of commutative ring spectra. Keeping $X$ free, we explore the geometric structure of our filtration. In the more technical Section \ref{mulstr}, we show that this filtration enjoys pleasant multiplicative properties. We then specialize to $X = S^1$, so that 
\[R \ot X \cong \THH(R).\]
Here, in Section \ref{Adams}, we identify the graded pieces of the filtration and show that they are eigenspectra for the Adams operations $\psi^r$, as they must be if we are to sensibly call our filtration a Hodge filtration.

In a forthcoming paper, we will describe how to lift the Hodge filtration on $\THH$ to a filtration by cyclotomic spectra, and thus obtain a filtration on topological cyclic homology $TC$. We also hope to show that the cyclotomic trace from $K$-theory to $TC$ makes the weight filtration on $K$-theory (various versions of which are explicated beautifully in \cite{Gra05}) compatible with the Hodge filtration on $TC$, ideally by showing that the trace is a map of spectral $\lambda$-rings (whatever these are) and by framing the filtrations on each side as $\gamma$-filtrations. In between, we plan to give explicit computations of the filtrations we define in interesting cases.

\section{Preliminaries on twisted arrow categories, ends and coends} \label{ends}
We'll first give a brief introduction to the definitions and basic properties of ends and coends of quasicategories; for a more thorough and classical treatment in the context of 1-categories, see \cite[Chapter IX]{Mac98}. Let $\mb{C}$ be a presentable symmetric monoidal $\iy$-category with colimits compatible with the symmetric monoidal structure - in particular, $\mb{C}$ has an internal hom right adjoint to the monoidal structure - and let $I$ be a small $\iy$-category. 

\begin{dfn} The \emph{twisted arrow category} $\twa_I$ is defined by
\[(\twa_I)_n = I_{2n + 1}\]
with faces and degeneracies given by
\[\tilde{d_i} x = d_{n - i} d_{n + 1  +  i} x\]
\[\tilde{s_i} x = s_{n - i} s_{n + 1 +  i} x.\]
We'll adopt the following unorthodox notation:
\[\twa^I := (\twa_I)^\op.\]
\end{dfn}
For a full discussion of twisted arrow categories, one should consult \cite[\S 4.2]{DAGX} or  \cite[\S2]{Bar14}, though bear in mind that these two sources differ by an $op$ in their definitions.

There is an evident functor $\mc{H}_I : \twa_I \to I^{op} \times I$ that takes an $n$-simplex $x$ of $\twa_I$ to $x|_{[0, n]} \times x|_{[n + 1, 2n + 1]}$. By \cite[Proposition 4.2.3]{DAGX} and \cite[Proposition 4.2.5]{DAGX}, it is a left fibration classified by the functor $I^\op \times I \to \Top$ that takes $X \times Y$ to $\Hom(X, Y)$. Let $\tau : (I^\op \X I)^\op \toe I^\op \X I$ be the natural equivalence, and let $\mc{H}^I$ denote the composite 
\[\mc{H}^I : \twa^I \os{\mc{H}_I^\op} \to (I^\op \X I)^\op \os{\tau} \to I^\op \X I.\]

If $T$ is any functor from $I^\op \times I$ to $\mb{C}$, we can obtain a functor $\twa_I \to \mb{C}$ by precomposing with $\mc{H}_I$ or a functor $\twa^I \to \mb{C}$ by precomposing with $\mc{H}^I$.
\begin{dfn} The \emph{coend} of $T$ is defined by
\[\int^I T := \colim{\twa^I} \mc{H}^I \circ T.\]
Dually, we define the \emph{end} of $T$ by
\[\int_I T := \llim{\twa_I} \mc{H}_I \circ T.\]
\end{dfn}
In particular, if $F : I^\op \to \mb{C}$ and $G : I \to \mb{C}$ are functors, we can define a functor $F \tilde{\times} G : I^\op \times I \to \mb{C} \times \mb{C}$ by taking products. 
The symmetric monoidal structure on $\mb{C}$ induces a functor $\otimes : \mb{C} \times \mb{C} \to \mb{C}$, and by postcomposition with $\otimes$ we obtain a functor
\[F \boxtimes G : I^\op \X I \to \mb{C}\]
(this is the objectwise tensor product of functors, not to be confused with the Day convolution product). The central current of this paper involves coends of the form
\[\int^I F \boxtimes G\]
which we'll abusively denote
\[\int^I F \otimes G,\]
but we'll first discuss an end of interest.

\begin{prop} \label{borlox}
Let $\mb{D}$ be any $\iy$-category and let $F, G : I \to \mb{D}$ be functors, both covariant this time. Let $\Hom_{\mb{D}} : \mb{D}^\op \times \mb{D} \to \Top$ be the hom-space functor (one model for this is mentioned above). Defining
 \[\Hom(F(-), G(-)) = \Hom_{\mb{D}} \circ ( F^{op} \X G) \circ \mc{H}_I : \twa_I \to \Top,\]
there is a canonical equivalence
\[\int_I \Hom(F(-), G(-)) \cong \Nat(F, G)\]
which is functorial in both $F$ and $G$.
\end{prop}
\begin{proof}
This statement, which is elementary in the 1-categorical setting, becomes slightly tricky to prove for $\iy$-categories, and we don't know of a proof elsewhere in the literature. We thank Clark Barwick and Denis Nardin for a helpful conversation regarding this proof. The reader actually interested in understanding this proof should first note that it really exists at the level of combinatorics, not homotopy theory: everything is going to be strictly defined, and we'll be fuelled by isomorphisms of simplicial sets rather than equivalences.

The first idea is that $\int_I \Hom(F(-), G(-))$ should be the category of ways of completing the three-quarters-of-a-commutative square

\[\xymatrix{\twa_I \ar[d]^{\mc{H}_I} & \twa_{\mb{D}} \ar[d]^{\mc{H}_{\mb{D}}} \\
I^\op \times I \ar[r]_{ F^{op}\X G} & \mb{D}^\op \times \mb{D}
}\]
to a commutative square; that is, it should be identified with the pullback

\[\Fun(\twa_I, \twa_{\mb{D}}) \X_{\Fun(\twa_I, \mb{D}^\op \times \mb{D})} \{( F^{op} \X G) \circ \mc{H}_I\}.\]
Indeed, let $S$ be the pullback

\[\xymatrix{ S \ar[rr] \ar[d] & & \twa_{\mb{D}} \ar[d]^{\mc{H}_{\mb{D}}} \\
\twa_I \ar[rr]_{( F^{op} \X G) \circ \mc{H}_I} & &  \mb{D}^\op \times \mb{D}.}\]
Then $S \to \twa_I$ is a left fibration classifying the composite $\Hom(F(-), G(-))$, and the space of sections of this fibration is a model for $\int_I \Hom(F(-), G(-))$. This means that in the diagram

\[\xymatrix{\int_I \Hom(F(-), G(-)) \ar[d] \ar[r] & \Fun(\twa_I, S) \ar[d] \ar[r] & \Fun(\twa_I, \twa_{\mb{D}}) \ar[d] \\
\{\id\} \ar[r] & \Fun(\twa_I, \twa_I) \ar[r] & \Fun(\twa_I, \mb{D}^\op \times \mb{D}),
}\]
the left-hand square is a pullback, and the right-hand square is certainly also a pullback. The composite pullback square gives the desired description of $\int_I \Hom(F(-), G(-))$.

Let $T$ denote the pullback
\[\begin{tikzcd}
T \ar{r} \ar{d} & \Fun(\twa_I, \twa_\mb{D}) \ar{d}{\mc{H}_\mb{D}} \\
\Fun(I^\op \X I, \mb{D}^\op \X \mb{D}) \ar{r}{\mc{H}_I} & \Fun(\twa_I, \mb{D}^\op \X \mb{D}),
\end{tikzcd}\]
so that $T$ is the category of commutative diagrams that look like
\[\begin{tikzcd}
\twa_I \ar{r} \ar{d}{\mc{H}_I} & \twa_\mb{D} \ar{d}{\mc{H}_\mb{D}} \\
I^\op \X I \ar{r} & \mb{D}^\op \X \mb{D}.
\end{tikzcd}\]
We claim that there's a pullback square
\[\begin{tikzcd}
\Nat(F, G) \ar{r} \ar{d}{\mc{H}_{\Fun(I, \mb{D})}} & T \ar{d} \\
\{(F^\op, G)\} \ar{r} & \Fun(I^\op \X I, \mb{D}^\op \X \mb{D}).
\end{tikzcd}\]
Indeed, one can give an explicit isomorphism between the fiber of $T$ over $(F^\op, G)$ and $\Nat(F, G)$. We'll give the bijection between the $0$-simplices; the higher simplices follow identically but with more cumbersome notation. The aim here is to give a bijection between the set $\Nat(F,G)_0$ of natural transformations from $F$ to $G$ and the set (call it $U$) of ways of extending $(F^\op, G)$ to a functor from $\twa_I$ to $\twa_{\mb{D}}$. 

First we'll define $\alpha : \Nat(F, G)_0 \to U$. An element $x \in \Nat(F, G)_0$ is in particular a map $x : \D^1 \X I \to \mb{D}$. If $y$ is a $k$-simplex of $\twa_I$ - thus a $2k+1$-simplex of $I$ - we'll define
\[\alpha(x)(y) = x(q_k \X y)\]
where $q_k : \D^{2k + 1} \to \D^1$ is the map with
\[q_k(i) = \begin{cases} 
0 & i \leq k \\
1 & i > k.
\end{cases}\]

The inverse $\beta$ of $\alpha$ is given as follows. Suppose $Q : \twa_I \to \twa_\mb{D}$ is a functor lifting $F^\op \X G$, and let $y = (y_{\D^1}, y_I)$ be a $k$-simplex of $\D^1 \X I$. There's a unique factorization
\[\begin{tikzcd}
\D^k \ar{r}{y} \ar{d} & \D^1 \X I \\
\D^{2k + 1} \ar{ru}[below, yshift=-5pt, near end]{(q_k, y'_I)}
\end{tikzcd}\]
where $y'_I$ is formed by applying some degeneracies to $y_I$, and the only purpose of the above diagram is to give a compact description of exactly which degeneracies these are. Then we'll define $\beta(Q)(y)$ to be the $k$-simplex of $\mb{D}$ resulting from removing the same degeneracies from the $(2k + 1)$-simplex $Q(y'_I)$. 

Once the definitions are fully unwrapped, it's immediate that $\alpha$ and $\beta$ land where they're supposed to, and it's easy to check that $\beta \circ \alpha$ and $\alpha \circ \beta$ are the identities. 

Composing the pullback square defining $T$ with the one we've just derived, we get a pullback square
\[\begin{tikzcd}
\Nat(F, G) \ar{r} \ar{d} & \Fun(\twa_I, \twa_\mb{D}) \ar{d} \\
\{(F, G)\} \ar{r} & \Fun(\twa_I, \mb{D}^\op \X \mb{D}).
\end{tikzcd}\]
But we already have a name for the fiber in this square: it's $\int_I \Hom(F(-), G(-))$. This gives us our isomorphism
\[\Nat(F, G) \cong \int_I \Hom(F(-), G(-)).\]
Moreover, this entire argument is contravariantly functorial in $I$. In particular, we could replace $I$ with $I \X \D^n$ for some $n$, which makes our equivalence functorial in $F$ and $G$.
\end{proof}
Now let $i : I \to J$ be a functor of small $\iy$-categories, and assume again that $\mb{C}$ is presentable. If $F : J \to \mb{C}$ is a functor, then we write $i^* F$ for $F \circ i$. Our next goal is to describe an ``adjunction" formula inolving $i^*$ and the left Kan extension $i_!^\op : \Fun(I^\op, \mb{C}) \to \Fun(J^\op, \mb{C}).$
\begin{prop}\label{coendind}
Let $G : I^{op} \to \mb{C}$ be a functor, and let $i_! G$ be a left Kan extension of $G$ along $i^{op}$; explicitly, on objects, 
\[i_! (G)(j) = \colim{z \in I \X_J J_{j/}} G(z).\]
Then there is an equivalence
\[\int^I i^* F \otimes G \simeq \int^J F \otimes i_! G\]
which is functorial in $F$ and $G$. 
\end{prop}

\begin{proof}
Let $T$ be a test object of $\mb{C}$. Then
\begin{align*}
 \Hom \left( \int^J F \otimes i_! G, T \right) & \simeq \int_J \Hom(i_! G \otimes F, T) \\
& \simeq \int_J \Hom(i_! G, \Hom(F(-), T)) \\
& \simeq \Nat(i_! G, \Hom(F(-), T)) \\
& \simeq \Nat(G, i^*\Hom(F(-), T)) \\
& \simeq \Nat(G, \Hom(i^* F(-), T)) \\
& \simeq \int_I \Hom(G, \Hom(i^*F(-), T)) \\
& \simeq \int_I \Hom(G \otimes i^*F, T) \\
& \simeq \Hom \left( \int^I i^* F \otimes G, T \right),
\end{align*}
and each equivalence is functorial in $F$ and $G$, which gives the result.
\end{proof}

\begin{prop} \label{constend}
For any functor $F : I \to \mb{C}$,
\[\int^I  \mathbbm{1}\ot F \simeq \colim I F\]
where $\mathbbm{1}$ is the constant functor at the unit of $\mb{C}$.
\end{prop}

\begin{proof}
We have
\[\int^I  \mathbbm{1} \ot F \simeq \colim{X \to Y \in \twa^I} F(X)\]
so it's enough to prove that the source map $s : \twa^I \to I$ is cofinal. By Joyal's version of Quillen's Theorem A \cite[Theorem 4.1.3.1]{HTT}, this is the same as proving that for each $X \in I$, the category $X_{/s} = \twa^I \X_I I_{X/}$ is weakly contractible as a space.

We may take as a model for $X_{/s}$ the simplicial set whose $n$-simplices are $2n + 2$-simplices of $I$ with leftmost vertex $X$. Let $X'_{/s}$ be the subcategory whose $n$-simplices are those $2n + 2$ simplices $\si$ of $I$ for which $\si|_{[0, n+ 1]}$ is totally degenerate at $X$. $X'_{/s}$ has the totally degenerate $2$-simplex at $X$ as a final object, so it suffices to show that the inclusion $X'_{/s} \inj X_{/s}$ is an equivalence. But the functor that sends a $2n + 2$ simplex $\si$ of $I$ to $s_0^{n + 1} d_1^{n + 1} \si$ is right adjoint to this inclusion.
\end{proof}

\section{Tensor products of commutative ring spectra with spaces}\label{tensor}

Let $\mc{F}$ be the category of finite sets and $\mc{F}_*$ its pointed counterpart, and let $E$ be a commutative ring spectrum. Given a simplicial set $X$, the tensor product or factorization homology $E \ot X$ is by definition the colimit of the constant $X$-diagram in $\CAlg$ valued at $E$; see \cite{GTZ} and \cite[\S 1]{MSV} for explorations of these ideas. We aim to give a topological version of the simplicial formula for $E \ot X$ in \cite[\S 3.1]{GTZ}. 

The following lemma is due to Barwick, and the proof is a near-verbatim reproduction of one communicated by him:
\begin{lem} \label{froj}
For $S$ a simplicial set, we denote by $\sSet_{/S}$ the category of simplicial sets over $S$ endowed with the covariant model structure \cite[Proposition 2.1.4.7]{HTT}. Suppose $j : S \to S'$ is a map of simplicial sets. Then the pullback functor $j^* : \sSet_{/S'} \to \sSet_{/S}$ and its left adjoint $j_!$ form a Quillen pair \cite[Proposition 2.1.4.10]{HTT}.

Now suppose
\[\begin{tikzcd}
Y' \ar{r}{g} \ar{d}{p'} & Y \ar{d}{p} \\
X' \ar{r}{f} & X
\end{tikzcd}\]
is a strict pullback square of simplicial sets in which $p$, and therefore $p'$, is smooth in the sense of \cite[Definition 4.1.2.9]{HTT}. Then the natural transformation
\[\mb{L}p'_! \circ \mb{R} g^* \to \mb{R} f^* \circ \mb{L}p_!\]
is an isomorphism of functors $\ho \, \sSet_{/Y} \to \ho \, \sSet_{/X'}.$
\end{lem}
\begin{proof}
Suppose $q: Z \to Y$ is a left fibration, and denote by $\ol{r} : \td{Z} \to X$ the fibrant replacement of $r := p \circ q$ in $\sSet_{/X}$; in particular, we have a covariant weak equivalence
\[\begin{tikzcd}[column sep = small]
Z \ar{rr}{\eta} \ar{rd}{r} & & \td{Z} \ar{ld}{\ol{r}} \\
& X.
\end{tikzcd}\]
By \cite[Proposition 4.1.2.15]{HTT}, both $r$ and $\ol{r}$ are smooth. Hence by \cite[Proposition 4.1.2.18]{HTT}, for any vertex $x \in X$, the induced map $\eta_x : Z_x \to \td{Z}_x$ s a weak equivalence of simplicial sets. Now since $p'$ is a smooth map as well, it follows that the natural map 
\[\begin{tikzcd}[column sep = small]
Z \X_X X' \ar{rr}{\eta'} \ar{rd}{r'} & & \td{Z} \X_X X' \ar{ld}{\ol{r}'} \\
& X'
\end{tikzcd}\]
is a covariant weak equivalence if and only if, for any point $\xi \in X'$, the induced map $\eta'_\xi : (Z \X_X X')_\xi \to (\td{Z} \X_X X')_\xi$ on the fiber over $x$ is a weak equivalence of simplicial sets. But this is true, since we can identify $(Z \X_X X')_\xi$ with $Z_{f(\xi)}$ and $(\td{Z} \X_X X')_\xi$ with $\td{Z}_{f(\xi)}$.
\end{proof}

\begin{cor}
Retaining the notation of Lemma \ref{froj}, let $\mb{C}$ be a presentable $\iy$-category and let $k : Y \to \mb{C}$ be a functor. Then we have an equivalence of functors $X' \to \mb{C}$
\[p'_! g^* k \toe f^* p_! k\]
where now $(-)^*$ is restriction and $(-)_!$ is left Kan extension.
\end{cor}

\begin{proof}
By straightening, the case $\mb{C} = \Top$ is equivalent to the statement of Lemma \ref{froj}. We can immediately extend this to presheaf categories $\mb{C} = \Fun(\mb{D}, \Top)$, since the square
\[\begin{tikzcd}
Y' \X \mb{D} \ar{r}{g} \ar{d}{p'} & Y \X \mb{D} \ar{d}{p} \\
X' \X \mb{D} \ar{r}{f} & X \X \mb{D}
\end{tikzcd}\]
also satisfies the hypotheses of Lemma \ref{froj} and since colimits in $\Fun(\mb{D}, \Top)$ are computed objectwise. Finally, composing $k$ with a localization functor doesn't change anything, and any presentable $\iy$-category is a localization of a presheaf category.
\end{proof}

Now we return to the problem of describing $E \ot X$ for a commutative ring spectrum $E$ and a simplicial set $X$, which we'll take to be finite. The functor $X : \D^\op \to \F$ classifies a left fibration $\td{X} \to \D^\op$ with finite set fibers, and $\td{X}$ is weakly equivalent to $X$. It fits into a pullback square
\[\begin{tikzcd}
\td{X} \ar{r}{g} \ar{d}{p'} & E \F \ar{d}{p} \\
\D^\op \ar{r}{X} & \F
\end{tikzcd}\]
where $E \F \to \F$ is the universal left fibration with finite set fibers, classified by the identity functor on $\F$. An object of $E \F$ is a finite set $S$ together with an element $s \in S$, and a morphism from $(S, s)$ to $(T, t)$ is a set map $f: S \to T$ with $f(s) = t$.

Let $k$ denote the constant functor $E \F \to \CAlg$ valued at $E$. Since $g^*k$ is also a constant functor, we have equivalences
\begin{align*}
E \ot X & \simeq E \ot \td{X} \\
& \simeq \colim{\D^\op} p'_! g^* k \\
& \simeq \colim{\D^\op} X^* p_! k.
\end{align*}
Let $U : \CAlg \to \Sp$ be the forgetful functor, and write $A_E := U \circ p_! k$. Since simplicial realizations of commutative ring spectra are computed on underlying spectra \cite[Corollary 3.2.3.2]{HA}, we have
\[E \ot X \simeq \colim{\D^\op} X^* A_E \tag{$\ast$}.\]
as spectra. This is our topological version of \cite[Definition 2]{GTZ}.

Note that $(\{1\}, 1)$ is an initial object of $E \F$, so that $k$ is left Kan extended from $(\{1\}, 1)$, and therefore $p_! k$ is left Kan extended from $\{1\} \in \F$. It has the property that $A_E(S) = E^{\wedge S}$. Note further that $A_E$ extends to a symmetric monoidal functor
\[A_E^\ot : \F^\amalg \to \Sp^\wedge\]
which is in turn the same data as $E$ itself, since a commutative ring spectrum is by definition a morphism of $\iy$-operads $\F_* \to \Sp^\wedge$ \cite[Definition 2.1.3.1]{HA} and $\F^\amalg$ is the symmetric monoidal envelope of $\F_*$ \cite[Construction 2.2.4.1]{HA}.

What can we do if there's a module in the mix? We'll use the following result, proved in \cite{Gla14a}:
\begin{thm}\label{fliy}
Let $\mb{C}^\ot$ be a symmetric monoidal $\iy$-category. We'll denote the category of finite sets by $\F$ and the category of finite pointed sets by $\F_*$. A datum comprising a commutative algebra $E$ in $\mb{C}$ and a module $M$ over it - that is, an object of $\Mod^{\F_*}(\mb{C})$, the underlying $\iy$-category of Lurie's $\iy$-operad $\Mod^{\F_*}(\mb{C})^\ot$ \cite[Definition 3.3.3.8]{HA} - gives rise functorially to a functor
\[A_{E, M} : \F_* \to \mb{C}\]
such that
\[A_{E, M}(S) \simeq E^{\ot S^o} \ot M.\]
\end{thm}
Let $X$ be a pointed finite simplicial set, thought of as a functor $\D^\op \to \F_*$. By analogy with $(\ast)$, we'll define the tensor product of $E$ with $X$ with coefficients in $M$ as
\[(E \ot X; M) := \colim{\D^\op} X^*A_{E, M}.\]
We expect this construction to agree, under appropriate circumstances, with the factorization homology of the pointed space $X$ regarded as a stratified space \cite{AFT14} with coefficients in a factorization algebra constructed from $E$ and $M$. 

We note also that a cocommutative coalgebra spectrum $P$ defines a functor $C_P : \mc{F}^{op} \to \Sp$, again mapping a finite set $S$ to the $S$-indexed smash power of $P$, but we won't go into the coherency details of this because all such functors arising in the present work can be defined easily at the point-set level.

\begin{exa}
When $X = S^1$ is the usual simplicial model for the circle, with
\[S^1[n] = \{0, 1, \cdots, n\}\]
$E \otimes S^1$ returns the usual simplicial expression for $THH(E)$. If $M$ is an $E$-module, then $(E \otimes S^1 ; M)$ returns the usual simplicial expression for $THH(E; M)$, where we point $S^1$ by the lone $0$-simplex \cite{MSV}.
\end{exa}

\section{The Hodge filtration} \label{Hodge}

We'll now exploit this lemmatic mass to derive a filtration of the topological Hochschild homology spectrum of a commutative ring spectrum $E$. More generally, this filtration exists on $E \ot X$ for an arbitrary simplicial set $X$. In this section, we'll define this filtration, present a geometric model for it, and give a convergence result. 

A key step is the following jugglement of coends:

\begin{prop}\label{naytya}
\begin{align*}
E \otimes X & = \colim{\D^{op}} X^* A_E \\
& \simeq \int^{\D^{op}}  \ul{\bb{S}} \wedge X^* A_E  & (\text{by Proposition }\ref{constend}) \\
& \simeq \int^\mc{F} X_! \ul{\bb{S}} \wedge A_E & (\text{by Proposition } \ref{coendind})
\end{align*}
where 
\begin{align*}X_! \ul{\bb{S}}(S) & \simeq \colim{[n] \in \D^{op}, S \to X[n]} \bb{S} \\
& \simeq \colim{[n] \in \D^{op}} \bb{S}^{\vee X[n]^S} \\
& \simeq \bb{S} \wedge \colim{\D^{op}} X[n]^S \\
& \simeq \bb{S} \wedge X^S \\
& \simeq \Si^\iy_+ X^S
\end{align*}
and the functoriality in $S$ is given by diagonal inclusions and projections; in other words, $X_! \ul{\bb{S}}$ is equivalent to the functor \[C_{\Si^\iy_+ X}: \mc{F}^{op} \to \Sp\]
coming from the coalgebra structure on $\Si^\iy_+ X$.
\end{prop} 
To lessen wrist fatigue, we'll usually wite $C_X$ for $C_{\Si^\iy_+ X}$.

Similarly, if $X_*$ is a pointed simplicial set, then
\[(E \otimes X_*; M) \cong \int^{\mc{F}_*}  X_{*, !} \ul{\bb{S}} \wedge A_{E, M}\]
where
\[X_{*, !} \ul{\bb{S}}(S) \cong \Si^\iy_+ X_*^{S^o}.\]
This is just \cite[Proposition 4.8]{Kuh04} in a slightly different language.

Now our filtration of $E \ot X$ will come from a filtration of the functor $C_X$.

\begin{dfn}
Let $\F^{\leq n}$ denote the full subcategory of $\F$ spanned by those sets $T$ for which
\[|T| \leq n.\]
Denote by $C^{\leq n}_X$ the restricted-and-extended functor
\[\Ran_{\F^{\leq n, \op}}^{\F^\op} C_X |_{\F^{\leq n, \op}}.\]
The $n$th Hodge-filtered quotient of $E \ot X$ is
\[\mf{H}^{\leq n}(E \ot X) := \int^{\F}  C^{\leq n}_X \wedge A_E.\]
In the same breath, we may as well define
\[\mf{H}^{\geq n + 1}(E \ot X) := \Fib \left[ E \otimes X \to \mf{H}^{\leq n}(E \ot X) \right]\]
and the graded subquotient
\[\mf{H}^{(n)}(E \ot X) := \Fib \left[ \mf{H}^{\leq n}(E \ot X) \to \mf{H}^{\leq n - 1}(E \ot X) \right].\]
Of course, if we define
\[C^{\geq n + 1}_X := \Fib \left[C_X \to C^{\leq n}_X \right]\]
and
\[C^{(n)}_X := \Fib \left[ C^{\leq n}_X \to C^{\leq n - 1}_X \right]\]
then we have the equivalences
\[\mf{H}^{\geq n + 1}(E \ot X) \simeq \int^\F  C^{\geq n + 1}_X \wedge A_E\]
and
\[\mf{H}^{(n)}(E \ot X) \simeq \int^\F  C^{(n)}_X \wedge A_E.\]
\end{dfn}
The story for the tensor product with coefficients is practically identical:
\begin{dfn}
 Let $\F_*^{\leq n}$ denote the full subcategory of $\F_*$ spanned by those objects $T$ for which $|T^o| \leq n$, and define
\[C^{\leq n}_{X_*} := \Ran_{\F_*^{\leq n, \op}}^{\F_*^\op} C_{X_*} |_{\F_*^{\leq n, \op}}.\]
Then the Hodge filtered quotient $\mc{H}^{\leq n}(E \ot X; M)$ is given by
\[\int^{\F_*} C^{\leq n}_{X_*} \wedge A_{E, M}.\]
\end{dfn}

\begin{rem}
Everything here goes through identically if we work in the category of algebras over a fixed commutative ring spectrum $K$, and we get a Hodge filtration of $E \ot_K X$ by $K$-modules.
\end{rem}

When $X = S^1$, we consider this the appropriate spectrum-level analogue of Loday's Hodge filtration (see \cite[4.5.15]{Lod92}.) We'll present evidence for this shortly, but first we'll showcase some of the geometry of this rather abstract-looking filtration. We'll abusively conflate $X$ with its geometric realization.

\begin{prop}\label{bromide}
Let $\F^{\inj}$ be the category of finite sets and injections and define a functor of 1-categories
\[\alpha_n : \F^{\inj, \op} \to \mc{T}op\]
by
\[\alpha_n(U) = \left \{ (x_u) \in X^U \ | \text{ at most }n \text{ coordinates of } (x_u) \text{ are not at the basepoint} \right \}\]
with the obvious projections as functorialities. Then the restriction of $C_X^{\leq n}$ to $\F^{\inj, \op}$ is equivalent to $\Si^\iy_+ \circ \alpha_n$. It's logical to set the convention $\alpha_\iy(U) = X^U$.
\end{prop}
\begin{proof}
First note that for each finite set $U$, the functor
\[\F^{\inj, \leq n}_{/U} \to \F^{\leq n}_{/U}\]
admits a left adjoint, and is thus homotopy cofinal. 

Now let $\mc{P}_U^{\leq n}$ be the poset of subsets of $U$ of cardinality at most $n$, $\mc{P}_U$ the poset of all subsets of $U$, and $\mc{P}'_U$ the poset of proper subsets of  $U$. The inclusion 
\[\mc{P}_U^{\leq n} \to \F^{\inj, \leq n}_{/U}\]
is an equivalence of categories, and $\Si^\iy_+ \circ \alpha_n$ and $C^{\leq n}_X$ clearly have equivalent restriction to $\mc{P}^{\leq n}_U$, so we are reduced to showing that
\[\Si^\iy_+ \alpha_n(U) \to \llim{V \in (\mc{P}_U^{\leq n})^\op} \Si^\iy_+ \alpha_n(V)\]
is an equivalence for every $U$. We work by induction on the cardinality of $U$; for $|U| \leq n$ there is nothing to prove. 
\begin{lem}
${\alpha_n}|_{\mc{P}_U}$ is a strongly cocartesian $U$-cube.
\end{lem}
\begin{proof}
Replace ${\alpha_n}|_{\mc{P}_U}$ with the diagram of cofibrations that takes $V \subseteq U$ to the subspace
\[\beta(V) = \left \{(x_u) \ | \ x_u \in X \text{ for } u \in V \right \} \] of  \[ \left \{ (x_u) \in (CX)^U \ | \text{ at most }n \text{ coordinates of } (x_u) \text{ are not at the basepoint} \right \}\]
where the basepoint on $CX$ is the basepoint on $X$, not the cone point. This is clearly homotopy equivalent to the original diagram, 
\end{proof}
So $V \mapsto \Si^\iy_+ \circ \alpha_n(V)$ is a cocartesian, and therefore cartesian, cube of spectra. But by the induction hypothesis, ${\alpha_n}|_{\mc{P}'_U}$ is right Kan extended from ${\alpha_n}|_{\mc{P}^{\leq n}_U}$, and so the entire cube is right Kan extended from ${\alpha_n}|_{\mc{P}^{\leq n}_U}$. Moreover, if $x \in \beta(V)$, then the subcube of $\mc{P}_U^{V \subseteq}$ spanned by those spaces containing $x$ is a face of $\mc{P}_U^{V \subseteq}$ containing $V$. This implies that the diagram is strongly cocartesian.
\end{proof}
\begin{cor} \label{nuaya}
$C^{\leq n}_X(U)$ is a retract of $C_X(U)$.
\end{cor}
\begin{proof}
Clearly ${\alpha_\iy}|_{\mc{P}_U}$ is strongly cartesian, and so $X^U$ is the limit of ${\alpha_\iy}|_{\mc{P}^{\leq n}_U} = {\alpha_n}|_{\mc{P}^{\leq n}_U}.$ This gives a map $\phi: \alpha_n(U) \to X^U$. On the other hand, the projection $\psi : C_X(U) \to C_X^{\leq n}(U)$ comes from regarding $\Si^\iy_+ \alpha_n(U)$ as the limit of $(\Si^\iy_+ \alpha_n)|_{\mc{P}^{\leq n}_U}$. Tracing through the universal properties shows that 
\[\psi \circ (\Si^\iy_+ \phi) : C^{\leq n}_X(U) \to C^{\leq n}_X(U)\]
is homotopic to the identity.
\end{proof}
\begin{cor}
Suppose $X$ is connected. Then the projection $\psi : C_X(U) \to C^{\leq n}_X(U)$ is $(n + 1)$-connected (by which we mean that the homotopy fiber of $\psi$ is $n$-connected).
\end{cor}
\begin{proof}
By Corollary \ref{nuaya}, it suffices to show that $\phi$ is $n$-connected. But if we model $X$ by a CW-complex with exactly one $0$-cell, which we take to be the basepoint, then $\phi$ is homotopic to the inclusion of the $n$-skeleton of the natural CW structure on $X^U$.
\end{proof}
\begin{cor}
Suppose $R$ is a connective commutative ring spectrum and $M$ is an $R$-module which is $k$-connective for some $k \in \bb{Z}$. Then the projection
\[(E \ot X ; M) \to (E \ot X ; M)^{\leq n}\]
is $(n + 1 + k)$-connected, and so we have convergence:
\[(E \ot X ; M) \simeq \llim{\leftarrow} (E \ot X; M)^{\leq n}.\]
\end{cor}
\begin{proof}
For each object $f : S \to T$ of $\twa^{\F_*}$, the projection
\[C_X(S) \wedge A_{E, M}(T) \to C_X^{\leq n}(S) \wedge A_{E, M}(T)\]
is $(n + 1 + k)$-connected. This connectivity is preserved by taking the colimit.
\end{proof}
\begin{rem}
We'd have to be born yesterday to expect good convergence behavior from our filtration in nonconnective situations.
\end{rem}

\section{Multiplicative structure}\label{mulstr}

The goal of this technical section is to pour some symmetric monoidality into our theory of coends; in particular, we'll see that taking a coend against a symmetric monoidal functor maps the Day convolution to the tensor product, thus evincing intriguingly Fourier transform-like behavior. This will allow us to deduce that the Hodge filtration on the tensor product of a space with a commutative ring is multiplicative (Corollary \ref{hanbo}).
 
\begin{lem} \label{prongo}
Let $\mb{C}^\ot, \mb{D}^\ot$ and $\mb{E}^\ot$ be symmetric monoidal $\iy$-categories, and let 
\[F : \mb{D}^\ot \to \mb{E}^\ot\]
be a symmetric monoidal functor which preserves colimits in each fiber. Then postcomposition with $F$ determines a symmetric monoidal functor
\[(\circ F) : \Fun(\mb{C}, \mb{D})^\ot \to \Fun(\mb{C}, \mb{E})^\ot.\] 
\end{lem}

\begin{proof}
Here's what this boils down to. Suppose $\mb{A}, \mb{B}$ and $\mb{C}$ are $\iy$-categories equipped with cocartesian fibrations to $\D^1$, and we're given functors $\mu : \mb{A} \to \mb{B}$ and $\nu : \mb{B} \to \mb{C}$
both compatible with the projections to $\D^1$. Suppose that $\nu$ preserves cocartesian edges and preserves colimits in the fibers, and moreover that $\mu$ is the relative left Kan extension of its restriction to the fiber $\mb{A}_0$ over $\{0\} \in \D^1$, by which we mean that the diagram
\[\begin{tikzcd}
\mb{A}_0 \ar{r}{\mu_0} \ar[hook]{d} & \mb{B} \ar{d}{p} \\
\mb{A} \ar[dotted]{ru}{\mu} \ar{r} & \D^1 
\end{tikzcd}\]
exhibits $\mu$ as a $p$-left Kan extension of $\mu_0$ in the sense of \cite[Definition 4.3.2.2]{HTT}. Then we want to show that the composition $\nu \circ \mu$ is also the relative left Kan extension of its restriction to $\mb{A}_0$.

Let $a$ be an object of $\mb{A}_1$, and denote 
\[(\mb{A}_0)_{/a} := \mb{A}_0 \X_\mb{A} \mb{A}_{/a}.\]
Denote by $\mu_a$ the natural map from $(\mb{A}_0)_{/a}$ to $\mb{B}$, and let 
\[\ol{\mu_a} :  (\mb{A}_0)_{/a} \X \D^1 \to \mb{B}\]
be an extension of $\mu_a$ to a morphism of cocartesian fibrations over $\D^1$. Taking fibers over $\{1\}$ gives a functor $\mu_a^1 : (\mb{A}_0)_{/a} \to \mb{B}_1$, and it follows from the proof of \cite[Corollary 4.3.1.11]{HTT} (see also \cite[Lemma 2.4]{Gla13}) that the condition on $\mu$ is equivalent to the condition that $\mu(a)$ is a colimit of $\mu_a^1$ for each $a \in \mb{A}_1$. But the condition on $\nu$ guarantees that $\nu \circ \mu$ satisfies these conditions if $\mu$ does.
\end{proof}
Let $I^\ot$ be a small symmetric monoidal $\iy$-category. Suppose we have a symmetric monoidal enrichment $(I^\op)^\ot$ of $I^\op$ and a symmetric monoidal enrichment $(\twa^I)^\ot$ of $\twa^I$ such that $\mc{H}^I$ extends to a symmetric monoidal functor
\[(\mc{H}^I)^\ot : (\twa^I)^\ot \to (I^\op)^\ot \X_{\F_*} I^\ot;\]
such enrichments are constructed in \cite{BGN}, but alternatively, if $I$ comes from is a symmetric monoidal 1-category, these objects are easily constructed as 1-categories, and all of the examples in the present work will be of this form.

In \cite[\S 2.3]{Lur14}, Lurie discusses a presentable stable symmetric monoidal $\iy$-category $\text{Rep}(I)^\ot$ together with a symmetric monoidal stable Yoneda embedding $I^\ot \to \text{Rep}(I)^\ot$ which induces an equivalence of categories
\[\Fun^{\ot, L} (\text{Rep}(I)^\ot, \mb{M}^\ot) \toe \Fun^\ot(I^\ot, \mb{M}^\ot),\]
where $\mb{M}^\ot$ is any stable presentable symmetric monoidal $\iy$-category and $\Fun^{\ot, L}$ is the category of symmetric monoidal functors which preserve colimits in the fibers.
\begin{prop}
The Day convolution category $\Fun(I^\op, \Sp)^\ot$ of \cite{Gla13} is equivalent as a symmetric monoidal $\iy$-category to the category $\text{Rep}(I)^\ot$. 
\end{prop}
\begin{proof}
Both categories certainly have underlying category $\Fun(I^\op, \Sp)$. By \cite[Remark 2.3.10]{Lur14}, we need only verify that $\Fun(I^\op, \Sp)^\ot$ satisfies the axioms characterizing $\text{Rep}(I)^\ot$. The first follows from \cite[Lemma 2.11]{Gla13}, and the second follows from \cite[Proposition 2.12]{Gla13} together with Lemma \ref{prongo}, since the suspension spectrum functor is symmetric monoidal and preserves colimits.
\end{proof}

Now let $F : I \to \Sp$ be a functor. Of course, $F$ factors essentially uniquely as

\[I \inj \Fun(I^\op, \Sp) \os{\td{F}} \to \Sp \]
where $\td{F}$ preserves colimits. We observe that $\td{F}$ is homotopic to the composite

\begin{align*}
\sigma: \Fun(I^\op, \Sp) & \os{\id \X F} \longrightarrow  \Fun(I^\op, \Sp) \X \Fun(I, \Sp) \\
& \longrightarrow \Fun(I^\op \X I, \Sp \X \Sp) \\
& \os{\wedge \circ (-) \circ \mc{H}^I} \longrightarrow \Fun(\twa^I, \Sp) \\
& \os{\text{colim}} \longrightarrow \Sp. 
\end{align*}
since $\sigma$ preserves colimits, and by the behavior of coends of representable functors, restricts to $F$ on $I$.

Suppose that $F$ extends to a symmetric monoidal functor $F^\ot : I^\ot \to \Sp^\wedge$. Then $\td{F}$ extends essentially uniquely to a symmetric monoidal functor $\td{F}^\ot : \Fun(I^\op, \Sp)^\ot \to \Sp^\wedge$ compatible with $F^\ot$. Thus coends against a symmetric monoidal functor $F$ behave like Fourier transforms: they interchange the Day convolution and the tensor product.

Now we'll specialize to the case where $I$ is the category $\F$ of finite sets. We'd like to give a symmetric monoidal enrichment of Proposition \ref{naytya}, which we'll build by composing several functors.

First, by \cite[Corollary 2.4.3.10]{HA}, we have an equivalence 
\[a : \Fun^\ot((\F^\op)^\amalg, \Top^\X) \toe \Top\].
Next, composition with the symmetric monoidal functor $(\Si^\iy_+)^\ot$ induces a functor 
\[s : \Fun^\ot((\F^\op)^\amalg, \Top^\X) \to \Fun^\ot((\F^\op)^\amalg, \Sp^\wedge).\]
We claim that $s$ is colimit-preserving; indeed, it suffices to show that $s$ preserves sifted colimits and coproducts. But $s$ preserves sifted colimits because these are computed objectwise in the target symmetric monoidal category, and it preserves coproducts because these are given by the tensor product and $\Si^\iy_+$ is a symmetric monoidal functor. (This argument will recur a couple of times in the next few paragraphs.)

Now let $\Alg_{\F^\op} (\Sp)$ be the category of \emph{lax} symmetric monoidal functors from $(\F^\op)^\amalg$ to $\Sp^\wedge$. We have an obvious full and faithful inclusion
\[l : \Fun^\ot((\F^\op)^\amalg, \Sp^\wedge) \to \Alg_{\F^\op}(\Sp).\]

\begin{lem}
$l$ preserves colimits.
\end{lem}
\begin{proof}
Once again, it's immediate that $l$ preserves sifted colimits. Now by \newline \cite{Gla13}[Proposition 2.10], we can identify $\Alg_{\F^\op}(\Sp)$ with the category 
\[\CAlg(\Fun(\F^\op, \Sp))\]
 of commutative algebras in the Day convolution symmetric monoidal category $\Fun(\F^\op, \Sp)^\ot$. Thus the coproduct in $\Alg_{\F^\op}(\Sp)$ is given by Day convolution, and what we want to show is that 
\begin{lem}
The Day convolution of two strict symmetric monoidal functors is canonically strict symmetric monoidal.
\end{lem}

Let's use the notation of the proof of \cite[Proposition 3.2.4.3]{HA}. We'll work in the model category $(\Set^+_\D)_{/\mf{P}'_{\Comm}}$, in which the fibrant objects are exactly symmetric monoidal categories with their cocartesian edges marked.

Let $\mu : \angs{2} \to \angs{1}$ be the active map. We regard $(\D^1)^\sharp \X (\F_*)^\sharp$ as a marked simplicial set over $(\F_*)^\sharp$ via the composite
\[(\D^1)^\sharp \X (\F_*)^\sharp \os{\mu \X \id} \to (\F_*)^\sharp \X (\F_*)^\sharp \os{\wedge} \to (\F_*)^\sharp.\]
Unwinding the definitions, we find that giving a pair of strict symmetric monoidal functors $\mb{C}^\ot \to \mb{D}^\ot$ together with a strict symmetric monoidal structure on their Day convolution is the same as giving a functor
\[((\D^1)^\sharp \X (\F_*)^\sharp) \X_{(\F_*)^\sharp} (\mb{C}^\ot)^\natural \to (\mb{D}^\ot)^\natural\]
of marked simplicial sets over $(\F_*)^\sharp$. Thus it suffices to show that the inclusion
\[i : (\{0\} \X (\F_*)^\sharp) \X_{(\F_*)^\sharp} (\mb{C}^\ot)^\natural \inj ((\D^1)^\sharp \X (\F_*)^\sharp) \X_{(\F_*)^\sharp} (\mb{C}^\ot)^\natural\]
is a trivial cofibration. But recall from the proof of \cite[Proposition 3.2.4.3]{HA} that there is a left Quillen bifunctor
\[\nu: (\Set^+_\D)_{/\mf{P}'_{\Comm}} \X (\Set^+_\D)_{/\mf{P}'_{\Comm}} \to (\Set^+_\D)_{/\mf{P}'_{\Comm}}\]
which takes a pair $(X, Y)$ to the product $X \X Y$ regarded as a marked simplicial set over $\F_*$ by composition with $\wedge$. Now the conclusion follows from the fact that $i$ arises as the product, under $\nu$, of $(\mb{C}^\ot)^\natural$ with the $\mf{P}'_\Comm$-anodyne inclusion 
\[\begin{tikzcd}[column sep=0]
\{\angs{2}\} \ar[hook]{rr} \ar{rd}  && (\D^1)^\sharp \ar{ld}{\mu} \\
& \F_*
\end{tikzcd}\]
\end{proof}
Finally, the functor $\td{F}^\ot$ induces a functor on commutative algebras, which we abusively denote
\[\td{F}^\ot : \Alg_{\F^\op}(\Sp) \to \CAlg(\Sp).\]
$\td{F}^\ot$ preserves colimits: it preserves sifted colimits, once again, because these are evaluated in $\Sp$, and it preserves coproducts because it is symmetric monoidal.

Consider the composite 

\[F_@ := \td{F}^\ot \circ l \circ s \circ a^{-1} : \Top \to \CAlg(\Sp).\]
$F_@$ is given by the formula
\[F_@(X) = \int^{\F} C_X \wedge F\]
and it extends uniquely to a symmetric monoidal functor from $\Top^\amalg$ to $\CAlg(\Sp)^\wedge$, since the symmetric monoidal structure is given by the coproduct on each side.

Observe that any $F$ is equivalent as a symmetric monoidal functor to $A_R$ for some commutative ring spectrum $R$. By the characterization of colimit-preserving functors from $\Top$ into a presentable $\iy$-category, we have proven the following:
\begin{thm}\label{any}
There is an equivalence of symmetric monoidal functors $\Top \to \CAlg(\Sp)$
\[\int^\F C_{(-)} \wedge A_R \simeq  R \ot (-).\]
\end{thm}
\begin{cor} \label{hanbo}
The Hodge filtration on $R \ot X$ is compatible with the multiplication in the sense that, in the diagram
\[\begin{tikzcd}
(R \ot X)^{\geq n} \wedge (R \ot X)^{\geq m} \ar[dotted]{r} \ar{d} & (R \ot X)^{\geq m + n} \ar{d} \\
(R \ot X) \wedge (R \ot X) \ar{r} & R \ot X,
\end{tikzcd}\]
the dotted arrow exists functorially in $R$.
\end{cor}
\begin{proof}
By Theorem \ref{any}, the diagram of solid arrows arises as the coend of $A_R$ against the diagram
\[\begin{tikzcd}
C_X^{\geq n} \ast C_X^{\geq m} \ar{d} & & C_X^{\geq m + n} \ar{d} \\
C_X \ast C_X \ar{r}{\sim} & C_{X \amalg X} \ar{r} & C_X,
\end{tikzcd}\]
where $\ast$ denotes Day convolution. But by definition, $C_X^{\geq m + n}$ is final among objects of $\Fun(\F^\op, \Sp)$ mapping to $C_X$ whose restriction to $(\F^\op)^{< m + n}$ is zero. $C_X^{\geq n} \ast C_X^{\geq m}$ shares this property, so the diagram can be completed in an essentially unique way.
\end{proof}

\section{The layers of the filtration and Adams operations}\label{Adams}

For the remainder of this paper, we'll stick with $X = S^1$. In this case we'll be able to give a description of the layers of our filtration, and as a side effect elucidate its compatibility with Adams operations.

Since $C^{\leq n}_{S^1}$ and (a fortiori) $C^{\leq n - 1}_{S^1}$ are right Kan extended from $\F^{\leq n, \op}$, so is the fiber $C^{(n)}_{S^1}$. Let us determine the homology, after restriction to $\F^{\leq n, \op}$, of $C^{(n)}_{S^1}$:

\begin{prop} \label{bonjox}
We have
\[H\bb{Z}_*C^{(n)}_{S^1}(U) = \begin{cases} \Si^n \bb{Z} & |U| = n \\ 0 & |U| < n \end{cases}\]
with $\Si_n$ acting by sign on $\Si^n \bb{Z}$.
\end{prop}
\begin{proof}
We need more notation. Let $T_0 := C^{\leq n}_{S^1}([n])$, the suspension spectrum of an $n$-torus, and let $T_1 := C^{\leq n - 1}_{S^1}([n])$, which is the suspension spectrum of a certain $(n-1)$-skeleton of an $n$-torus, as described in Proposition \ref{bromide}.

Clearly the projection
\[z : C^{\leq n}_{S^1} \to C^{\leq n - 1}_{S^1}\]
is an equivalence when evaluated on sets of cardinality less than $n$, so it suffices to determine the fiber of $z$ on $\F^{(n), \op}$. Denote this fiber $L$; certainly $H_n(L)$ is $\bb{Z}$ with the sign action of $\Si_n$, since this is $H_n(T_0)$. We aim to show that $H_i(L) = 0$ for $i < n$.

 For each $i < n$, we have to show that the map
\[z_i : H_i(T_0) \to H_i(T_1)\]
is an isomorphism. Observe that both groups are free of rank $\binom n i$. For every injective map $g: [i] \to [n]$, we get projections
\[T_0, T_1 \to \Si^\iy_+ (S^1)^i\]
which on $H_i$, by Proposition \ref{bromide}, induce the projection onto the $\bb{Z}$-factor corresponding to $g$. The diagram
\[\begin{tikzcd}
T_0 \ar{r} \ar{d} & T_1 \ar{ld} \\
\underset{U \inc [n], |U| = i} \bigvee \Si^\iy_+ (S^1)^U 
\end{tikzcd}\]
homotopy commutes by functoriality, and on taking the $i$th homology we get the commutative diagram
\[\begin{tikzcd}
H_i(T_0) \ar{r}{z_i} \ar{d} & H_i(T_1) \ar{ld} \\
\underset{U \inc [n], |U| = i} \bigvee \bb{Z}.
\end{tikzcd}\]
The vertical and diagonal maps are isomorphisms, so that $z_i$ is an isomorphism. This completes the proof.
\end{proof}
Continuing the notation of Proposition \ref{bonjox}, we'll write $L$ for the restriction of $C^{(n)}_{S^1}$ to $\F^{\op, \leq n}$.
\begin{lem} \label{yutonx}
The homology of $L$ determines it up to equivalence.
\end{lem}
\begin{proof}
Let 
\[L' : \F^{\op, \leq n} \to \Sp\]
be another functor together with an isomorphism between $H \bb{Z} \wedge L'$ and $H \bb{Z} \wedge L$. We claim that there is an equivalence
\[H \bb{Z} \wedge \Nat(L', L) \to \Nat(H \bb{Z} \wedge L', H \bb{Z} \wedge L).\]
 Indeed, by Proposition \ref{borlox}, we have maps
\begin{align*} H \bb{Z} \wedge \Nat(L', L) & \simeq H \bb{Z} \wedge \int_{\F^{\op, \leq n}} \Hom(L'(-), L(-)) \\
& \simeq  \int_{\F^{\op, \leq n}} H \bb{Z} \wedge \Hom(L'(-), L(-)) \\ 
& \to  \int_{\F^{\op, \leq n}} \Hom(H \bb{Z} \wedge L'(-), H \bb{Z} \wedge L(-)) \\
& \simeq \Nat(H \bb{Z} \wedge L', H \bb{Z} \wedge L). \end{align*}
Since $L'$ and $L$ take values only in shifted spheres and contractible spectra, the map
\[H \bb{Z} \wedge \Hom(L'(S), L(T)) \to \Hom(H \bb{Z} \wedge L'(S), H \bb{Z} \wedge L(T))\]
is an equvalence for all $S$ and $T$, so all maps in the chain are in fact equivalences. That means we have a class in $H \bb{Z}_0 \Nat(L', L)$ corresponding to the given isomorphism between $H \bb{Z} \wedge L'$ and $H \bb{Z} \wedge L$. Any natural transformation in the corresponding connected component induces an isomorphism on integral homology, and so is an equivalence.
\end{proof}

This makes it easy to write down a point-set model for $L$, if one is so inclined. We also get the following result, which will be useful in just a minute:

\begin{cor} \label{banvox}
Any natural endomorphism of $L$ which acts by multiplication by an integer $s$ on homology is equivalent to the multiplication-by-$s$ endomorphism on $L$.
\end{cor}
\begin{proof}
The proof is the same as that of Lemma \ref{yutonx}.
\end{proof}

On to Adams operations. The formula 

\[\THH(R) = R \ot S^1\]
immediately suggests a family of potentially interesting operations
\[\{\psi^r \, | \, r \in \bb{Z}\}\]
given by
\[\psi^r = \id \ot c_r : R \ot S^1 \to R \ot S^1\]
where $c_r$ is a degree $r$ endomorphism of $S^1$. These are called the \emph{Adams operations}, and one can refer to \cite[\S 11]{ABGHLM} for a detailed account of their properties and history.

From our coend formula point of view, $\psi^r$ arises from the natural transformation

\[[r] : C_{S^1} \to C_{S^1}\]
acting by multiplication by $r$ on each torus. Note that on the top homology of an $n$-torus, $r$ acts by multiplication by $r^n$. By naturality of Kan extensions, $[r]$ descends to a natural transformation
\[[r]^{\leq n} : C_{S^1}^{\leq n} \to C_{S^1}^{\leq n}\]
and thus a natural operation
\[(\psi^r)^{\leq n} : \THH^{\leq n} \to \THH^{\leq n}\]
Moreover, by naturality of the formation of homotopy fibers, we have a homotopy commutative diagram

\[ \begin{tikzcd}
C_{S^1}^{(n)} \ar{r} \ar{d}{[r]^{(n)}} & C_{S^1}^{\leq n} \ar{r} \ar{d}{[r]^{\leq n}} & C_{S^1}^{\leq n- 1} \ar{d}{[r]^{\leq n -1}} \\
C_{S^1}^{(n)} \ar{r}  & C_{S^1}^{\leq n} \ar{r} & C_{S^1}^{\leq n- 1}  
 \end{tikzcd} \]
 
The induced natural transformation $[r]_L : L \to L$ acts by multiplication by $r^n$ on homology, and by Corollary \ref{banvox}, $[r]_L$ is multiplication by $r^n$. Since $C_{S^1}^{(n)}$ is right Kan extended from $L$, we have
\[[r] = r^n : C_{S^1}^{(n)} \to C_{S^1}^{(n)}\]
and so:
\begin{prop}
Defining 
\[(\psi^r)^{(n)} : \THH^{(n)} \to \THH^{(n)}\]
by naturality of homotopy fibers, we have
\[(\psi^r)^{(n)} = r^n.\]
\end{prop}
Loday's Hodge filtration \cite{Lod89} is defined as the $\gamma$-filtration associated to a $\lambda$-ring structure, and this kind of compatibility with the Adams operations is a prominent characteristic of any $\gamma$-filtration \cite[Proposition 3.1]{FL}. It should thus be a prerequisite for any putative Hodge filtration on $\THH$.
\bibliographystyle{alpha}
\bibliography{Mybib}

\end{document}